\documentclass[12pt,twoside]{article}
\usepackage{amssymb,amsmath,amsthm,latexsym}
\usepackage{amsfonts}
\usepackage{enumerate}
\usepackage{subcaption}
\usepackage{relsize}
\usepackage{xcolor}
\usepackage{hhline}
\usepackage{mathrsfs}
\usepackage{doi}
\usepackage{tikz}
\usetikzlibrary{shapes, arrows, calc, arrows.meta, fit,quotes, positioning}
\usepackage{calc}
\usetikzlibrary{decorations.markings}
\tikzstyle{vertex}=[circle, draw, inner sep=0pt, minimum size=6pt]
\newcommand{\vertex}{\node[vertex]}
\newtheorem{thm}{Theorem}[section]

\newtheorem{lem} [thm]{Lemma}
\newtheorem{prop} [thm]{Proposition}
\theoremstyle{definition} 

\newtheorem{ex}[thm]{Example}
\newtheorem{rmk}[thm] {Remark}
\newtheorem{defn}[thm]{Definition}
\newtheorem{obs}[thm]{Observation}
\numberwithin{equation}{section}

\begin{document}
	\label{'ubf'}
	\setcounter{page}{1} 

	\markboth {\hspace*{-9mm} \centerline{\footnotesize \sc
			Mycielskian of Signed Graphs}
	}
	{ \centerline {\footnotesize \sc  Albin,  Germina
			}
	}
	\begin{center}
		{
			\Large \textbf{Mycielskian of  Signed Graphs
				}
			}

			\bigskip
		Albin Mathew \footnote{\small Department of Mathematics, Central University of Kerala, Kasaragod - 671316,\ Kerala,\ India. \textbf{Email:}\texttt{ albinmathewamp@gmail.com, }}\, Germina K A \footnote{\small Department of Mathematics, Central University of Kerala, Kasaragod - 671316,\ Kerala,\ India. \textbf{Email:}\texttt{ srgerminaka@gmail.com}}\,
			\\

\end{center}

\thispagestyle{empty}
\begin{abstract}
 In this paper, we define the Mycielskian of a signed graph and discuss the properties of  balance and switching in the Mycielskian of a given signed graph. We provide a condition for ensuring the Mycielskian of a balanced signed graph remains balanced, leading to the construction of a balanced Mycielskian. We establish a relation between the chromatic numbers of a signed graph and its Mycielskian. We also study the structure of different matrices related to the Mycielskian of a signed graph.
\end{abstract}




\section{Introduction}\label{sec1}

A signed graph $\Sigma=(G,\sigma)$ consists  of an underlying graph $G=(V,E)$, together with a function $\sigma:E\rightarrow \{-1,1\}$, called the signature or sign function. The sign of a cycle in a signed graph is the product of the signs of  its edges. A signed graph $\Sigma$ is said to be balanced if no negative cycles exist, otherwise $\Sigma$ is unbalanced. A signed graph is called all-positive (all-negative)  if all the edges are positive (negative).

A switching function for $\Sigma$ is a function $\zeta:V(\Sigma)\rightarrow\{-1,1\}$. For an edge $e=uv$ in $\Sigma$, the switched signature $\sigma^\zeta$ is defined as $\sigma^\zeta(e)=\zeta(u)\sigma(e)\zeta(v)$,  and the switched signed graph is  $\Sigma^\zeta=(G,\sigma^\zeta)$  (see \cite[Section 3]{tz}). The signs of cycles are unchanged by switching, and any balanced signed graph can be switched to an all-positive signed graph. If one signed graph can be switched from the other, they are said to be switching equivalent. Two signed graphs $\Sigma_1$ and $\Sigma_2$ are said to be switching isomorphic if $\Sigma_1$ is isomorphic to a switching of $\Sigma_2$.

The net-degree of a vertex $v$ in a signed graph $\Sigma$, denoted by $d_\Sigma^\pm(v)$ is defined as $d_\Sigma^\pm(v)=d_\Sigma^+(v)-d_\Sigma^-(v)$, where $d_\Sigma^+(v)$ and $d_\Sigma^-(v)$ respectively denotes the number of positive and negative edges incident with $v$ in $\Sigma$. The total number of edges incident with $v$ in $\Sigma$ is denoted by $d_\Sigma(v)$ and $d_\Sigma(v)=d_\Sigma^+(v)+d_\Sigma^-(v)$.

Throughout this paper, we consider only finite, simple, connected and undirected graphs and signed graphs. For the standard notation and terminology in graphs  and signed graphs  not given here, the reader may refer to \cite{fh} and \cite{tz1,tz3} respectively.

The Mycielski construction of a simple graph was introduced by J. Mycielski \cite{my} in his search for  triangle-free graphs with arbitrarily large chromatic number. The Mycielskian for a finite, simple, connected graph $G=(V,E)$ is defined as follows.
\begin{defn}\label{def0}\cite{br}
	The Mycielskian $M(G)$ of $G$ is a graph whose vertex set is the disjoint union $V\cup V'\cup \{w\}$, where $V'=\{v':v\in V\}$, and whose edge set is $E\cup \{u'v:uv\in E\}\cup \{v'w:v'\in V'\}$. The vertex $w$ is called the root of $M(G)$ and $v'\in V'$ is called the twin of $v$ in $M(G)$.
\end{defn}

\subsection{Mycielskian of signed  graphs}\label{sec1.1}
  Motivated from the Definition~\ref{def0}, we define the Mycielskian $ M(\Sigma)$ of the signed graph $\Sigma$ as follows.
\begin{defn}[Mycielskian]\label{def1}
		The Mycielskian of $\Sigma$ is the signed graph  $M(\Sigma)=(M(G),\sigma_M)$, where $M(G)$  is the Mycielskian of the underlying graph $G$ of $\Sigma$, and the signature function $\sigma_M$ is defined as $\sigma_M(uv)=\sigma_M(u'v)=\sigma(uv)$ and $\sigma_M(v'w)=1$
\end{defn}
The following are some immediate observations.
\begin{obs}
	Let $\Sigma$ be a signed graph with $p$ vertices and $q$ edges and let $M(\Sigma)$ be its Mycielskian. Then, we have the following.
	\begin{enumerate}
		\item [\rm{(i)}] $M(\Sigma)$ has $2p+1$ vertices and $3q+p$ edges.
		\item [\rm{(ii)}] If $\Sigma$ contains $r$ positive edges and $q-r$ negative edges, then $M(\Sigma)$ contains $3r+p$ positive  edges and $3(q-r)$ negative edges.
		\item[\rm{(iii)}] If $\Sigma$ is  triangle-free, then $M(\Sigma)$ is also triangle-free.
		\item [\rm{(iv)}] For each vertex $v\in V$, $d^\pm_{M(\Sigma)}(v)=2d^\pm_\Sigma(v)$ and $d_{M(\Sigma)}(v)=2d_\Sigma(v)$.
		\item [\rm{(v)}] For each vertex $v'\in V'$, $d^\pm_{M(\Sigma)}(v')=d^\pm_\Sigma(v)+1$ and $d_{M(\Sigma)}(v')=d_\Sigma(v)+1$ .
		\item [\rm{(vi)}] $d^\pm_{M(\Sigma)}(w)=d_{M(\Sigma)}(w)=p$.
	\end{enumerate}
\end{obs}

Note that one can define the signature function for the Mycielskian of a signed graph in other ways. In this paper, we initiate a study on Mycielskian of a signed graph using this particular definition.

This particular construction of Mycielskian of a signed graph is illustrated in Example~\ref{ex1}.

\begin{ex}\label{ex1}	Let $\Sigma$ be the negative cycle $C_4^-$. The Mycielskian of  $C_4^-$ is constructed in Figure \ref{fig:sub2}.
	\begin{figure}[h!]	\centering
		\begin{subfigure}{.5\textwidth}
			\centering
			\begin{tikzpicture}[x=0.55cm, y=0.55cm]
				\vertex[fill] (v) at (0,0) [label=left:$v_1$] {};
				\vertex[fill] (w) at (4,0) [label=right:$v_2$] {};
				\vertex[fill] (x) at (0,4) [label=left:$v_4$] {};
				\vertex[fill] (y) at (4,4) [label=right:$v_3$] {};
				\path[very thick, dotted]
				(v) edge (w)

				;
				\path
				(x) edge (v)
				(w) edge (y)
				(y) edge (x)
				;
			\end{tikzpicture}
			\caption{$\Sigma$}
			\label{fig:sub1}
		\end{subfigure}%
		\begin{subfigure}{.5\textwidth}
			\centering
			\begin{tikzpicture}[x=0.55cm, y=0.55cm]
				\vertex[fill] (v1) at (0,0) [label=left:$v_1$] {};
				\vertex[fill] (v2) at (2,0) [label=above:$v_2$] {};
				\vertex[fill] (v3) at (4,0) [label=above:$v_3$] {};
				\vertex[fill] (v4) at (6,0) [label=right:$v_4$] {};
				\vertex[fill] (v1') at (0,-2) [label=left:$v_1'$] {};
				\vertex[fill] (v2') at (2,-2) [label=right:$v_2'$] {};
				\vertex[fill] (v3') at (4,-2) [label=right:$v_3'$] {};
				\vertex[fill] (v4') at (6,-2) [label=right:$v_4'$] {};
				\vertex[fill] (w) at (3,-4) [label=below:$w$] {};

				\path[very thick, dotted]
				(v1) edge (v2)
				(v1) edge (v2')
				(v1') edge (v2)
				;
				\path
				(v2) edge (v3)
				(v3) edge (v4)
				(v2') edge (v3)
				(v2) edge (v3')
				(v3') edge (v4)
				(v3) edge (v4')
				(v1') edge (v4)
				(v1) edge (v4')
				(v1') edge (w)
				(v2') edge (w)
				(v3') edge (w)
				(v4') edge (w)
				;
				\draw[] (0,0) .. controls (3,2.3) .. (6, 0);
			\end{tikzpicture}
			\caption{$M(\Sigma)$}
			\label{fig:sub2}
		\end{subfigure}
		\caption{A signed graph and its Mycielskian.}
		\label{fig:test1}
	\end{figure}
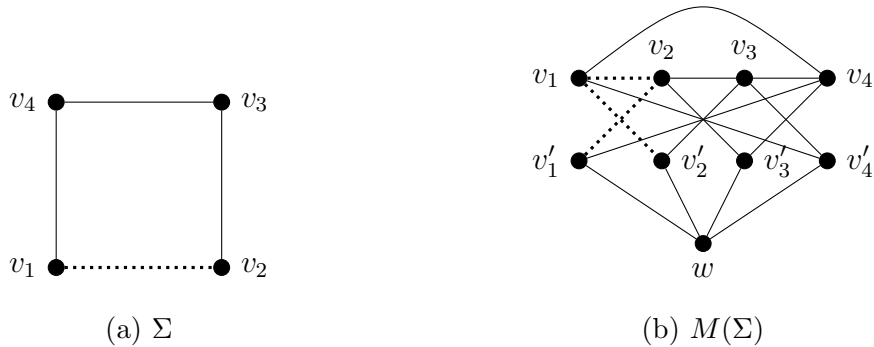
\end{ex}

\section{Balance and switching in Mycielskian of  signed graphs}\label{sec2}
Balance and switching are two important concepts in signed graph theory.

In this section, we establish how the signed graph and its Mycielskian are related with respect to balance and switching.
One may note that if $\Sigma$ is unbalanced, then $M(\Sigma)$ is unbalanced. Also, in general, for a balanced signed graph $\Sigma$, the Mycielskian $M(\Sigma)$ need not be balanced.

The following is a characterization for $M(\Sigma)$ to be balanced.

\begin{prop}\label{ag4}
	The Mycielskian $M(\Sigma)$ is balanced if and only if $\Sigma$ is all-positive.
\end{prop}
\begin{proof}
	If $\Sigma$ is all-positive, then so is $M(\Sigma)$, and hence is balanced. Conversely, If $\Sigma$ has at least one negative edge, say $v_iv_j$, then $v_iv_jv_i'w  v_j'v_i$ forms a negative $5$ - cycle in $M(\Sigma)$, making it unbalanced.
\end{proof}

Consider any balanced signed graph $\Sigma$ which is not all-positive. Then $\Sigma$ can be switched to an all-positive signed graph, say $\Sigma'$. By Proposition~\ref{ag4}, $M(\Sigma)$ is not balanced, but $M(\Sigma')$ is balanced. Hence, the Mycielskians of two switching equivalent signed graphs need not to be switching  equivalent.\\

The Mycielskian of an unbalanced signed graph is always unbalanced. However, for a balanced signed graph $\Sigma$, the Mycielskian $M(\Sigma)=(M(G),\sigma_M)$ can be made balanced by modifying the signature function $\sigma_M$. Though there are several ways to do so, to remain consistent with our original definition, we only look for changes that can be made in the signature of the edges incident to the root vertex $w$ which makes the Mycielskian balanced, and leave the signatures of the other edges unchanged.

We need the following theorem \cite{fh1}.
\begin{thm}[Harary's bipartition theorem \cite{fh1}]\label{Harary}
	A signed graph $\Sigma$ is balanced if and only if there is a bipartition of its vertex set, $V = V_1 \cup V_2$, such that every positive edge is induced by $V_1$ or $V_2$ while every negative edge has one endpoint in $V_1$ and one in $V_2$. The bipartition $V = V_1 \cup V_2$ is called a Harary bipartition for $\Sigma$.
\end{thm}

Note that if $V = V_1 \cup V_2$ is a Harary bipartition for $\Sigma$, then every path in $\Sigma$ joining vertices in $V_1$ (similarly $V_2$) is positive, and every path between $V_1$ and $V_2$ is negative.

Theorem~\ref{ag1} provides a method to construct a balanced Mycielskian signed graph from a balanced signed graph.
\begin{thm}\label{ag1}
	Let $\Sigma$ be a balanced signed graph and $M(\Sigma)=(M(G),\sigma_M)$ be its Mycielskian. If $\sigma_M'$ is a signature function satisfying $\sigma_M'=\sigma_M$ on $M(G)\backslash \{w\}$ and  satisfies the relation $\sigma_M'(v_i'w)\sigma_M'(v_j'w)=\sigma(v_iv_j)$ for every edge $v_iv_j$ in $\Sigma$, then  the signed graph $M'(\Sigma)=(M(G),\sigma'_M)$ is balanced.
\end{thm}

\begin{proof}
	Since $\Sigma$ is balanced, by Harary bipartition theorem, there exist a bipartition $V=V_1\cup V_2$ of $V$ such that every negative edge in $\Sigma$ has its one end vertex in $V_1$ and the other in $V_2$. We construct a Harary bipartition for $M'(\Sigma)$ as follows.

	For $i=1,2$, let $V_i'=\{v_i':v_i\in V_i\}$ be the subsets of $V'$ corresponding to the subsets $V_1$ and $V_2$ of $V$. Since $V = V_1 \cup V_2$, we have $V' = V_1' \cup V_2'$. Every edge with both its end vertices in $V_1$ is positive and no vertices in $V_1'$ are adjacent. Also, for edges $v_iv_j'$, where $v_i\in V_1$ and $v_j'\in V_1'$, $\sigma_M'(v_iv_j')=\sigma_M(v_iv_j')=\sigma(v_iv_j)=+1$. Thus, every edge with both its end vertices in $V_1\cup V_1'$ is positive. Similarly, every edge with both its end vertices in $V_2\cup V_2'$ is positive.

	Consider any edge $e$ having one end vertex in $V_1\cup V_1'$ and the other in $V_2\cup V_2'$. There are three possibilities.
	\begin{enumerate}
		\item If $e=v_iv_j$, where $v_i\in V_1$ and $v_j\in V_2$, then $\sigma_M'(e)=\sigma_M(e)=\sigma_M(v_iv_j)=\sigma(v_iv_j)=-1$.
		\item If $e=v_iv_j'$, where $v_i\in V_1$ and $v_j'\in V_2'$, then $\sigma_M'(e)=\sigma_M(e)=\sigma_M(v_iv_j')=\sigma(v_iv_j)=-1$.
		\item If $e=v_i'v_j$, where $v_i'\in V_1'$ and $v_j\in V_2$, then $\sigma_M'(e)=\sigma_M(e)=\sigma_M(v_i'v_j)=\sigma(v_iv_j)=-1$.
	\end{enumerate}
	Hence, every edge joining $V_1\cup V_1'$ and $V_2\cup V_2'$ is negative.

	We claim that : If $\sigma_M'(v_k'w)$ is positive for some $v_k\in V_1$, then $\sigma_M'(v_i'w)$ is positive for all $v_i\in V_1$ and $\sigma_M'(v_j'w)$ is negative for all $v_j\in V_2$.

	To prove the claim, note that if  $\sigma_M'(v_i'w)\sigma_M'(v_j'w)=\sigma(v_iv_j)$ for every edge $v_iv_j$ in $\Sigma$, the same holds for every $v_iv_j$ path in $\Sigma$. For, consider a $v_iv_j$ path, say $v_iv_{i+1}v_{i+2}\cdots v_{j-1}v_{j}$, in $\Sigma$. Then,
	\begin{align*} \sigma(v_iv_j)&=\sigma(v_iv_{i+1}v_{i+2}\cdots v_{j-1}v_{j})\\
		&=\sigma(v_{i}v_{i+1})\sigma(v_{i+1}v_{i+2})\cdots\sigma(v_{j-1}v_{j})\\
		&=(\sigma_M'(v_{i}'w)\sigma_M'(v_{i+1}'w))(\sigma_M'(v_{i+1}'w)\sigma_M'(v_{i+2}'w))\cdots(\sigma_M'(v_{j-1}'w)\sigma_M'(v_{j}'w))\\
		&=\sigma_M'(v_{i}'w)(\sigma_M'(v_{i+1}'w)\sigma_M'(v_{i+2}'w)\cdots\sigma_M'(v_{j-1}'w))^2\sigma_M'(v_{j}'w))\\
		&=\sigma_M'(v_{i}'w)\sigma_M'(v_{j}'w).
	\end{align*}

	Now, consider  $v_k\in V_1$ and let $v_i\in V_1$ and $v_j\in V_2$ be arbitrary. Then every $v_iv_k$ path is positive and every $v_jv_k$ path is negative. The connectedness of $\Sigma$ guarantees the existence of such paths.\\
	Now, $\sigma_M'(v_i'w)\sigma_M'(v_k'w)=\sigma(v_iv_k)=+1$. Thus. $\sigma_M'(v_i'w)$ and $\sigma_M'(v_k'w)$ must have the same sign. \\
	Similarly, since $\sigma_M'(v_j'w)\sigma_M'(v_k'w)=\sigma(v_jv_k)=-1$, $\sigma_M'(v_j'w)$ and $\sigma_M'(v_k'w)$ are of the opposite sign.

	Thus, if $\sigma_M'(v_k'w)$ is positive for some $v_k\in V_1$, then $\sigma_M'(v_i'w)$ is positive for all $v_i\in V_1$ and $\sigma_M'(v_j'w)$ is negative for all $v_j\in V_2$. Hence, the claim is proved.

	Now consider the edges $v_i'w$, where, $v_i'\in V_1'\cup V_2'$. Because of the claim,  if $\sigma_M'(v_k'w)$ is positive for some $v_k\in V_1$, then $\sigma_M'(v_i'w)$ is positive for all $v_i\in V_1$ and $\sigma_M'(v_j'w)$ is negative for all $v_j\in V_2$. \\
	In this case, take $(V_M)_1=V_1\cup V_1'\cup\{w\}$ and $(V_M)_2=V_2\cup V_2'$.\\
	Similarly, if
	$\sigma_M'(v_k'w)$ is negative for some $v_k\in V_1$, then $\sigma_M'(v_i'w)$ is negative for all $v_i\in V_1$ and $\sigma_M'(v_j'w)$ is positive for all $v_j\in V_2$. \\
	In this case, take $(V_M)_1=V_1\cup V_1'$ and $(V_M)_2=V_2\cup V_2'\cup\{w\}$.

	Thus, in either cases,
	$V_M=(V_M)_1\cup (V_M)_2$  forms a  Harary bipartition for $M'(\Sigma)$, and hence
	$M'(\Sigma)$ is balanced.
\end{proof}
\begin{rmk}
One may note that $\sigma_M'$ is a different signature on $M(G)$ that coincides with $\sigma_M$ on $M(G)\backslash\{w\}$. The signature function $\sigma_M'$ for the remaining edges $v_i'w$ of $M(G)$ has to be defined using the relation stated in Theorem~\ref{ag1}. One such construction is discussed in Section~\ref{bc}.

It is also worth noting that if $\sigma_M'=\sigma_M$ on $M(G)$, then Theorem~\ref{ag1} reduces to Proposition~\ref{ag4}.
\end{rmk}
\subsection{A balance-preserving construction}\label{bc}
Given any balanced signed graph $\Sigma=(G,\sigma)$, there exist a switching function $\zeta:V(\Sigma)\rightarrow \{-1,+1\}$ that switches $\Sigma$ to all-positive. Define $M_B(\Sigma)$ as the signed graph with underlying graph $M(G)$ and having the signature function $\sigma_B$ defined as
\begin{align*}
	\sigma_B(v_iv_j)&=\sigma(v_iv_j),\\
	\sigma_B(v_i'v_j)&=\sigma_B(v_iv_j')=\sigma(v_iv_j),\\
	\sigma_B(v_i'w)&=\zeta(v_i).
\end{align*}
Define a switching function $\zeta_B:V(M_B(\Sigma))\rightarrow\{-1,+1\}$ by
\begin{align*}
	\zeta_B(v_i)&=\zeta(v_i),\\
	\zeta_B(v_i')&=\zeta(v_i),\\
	\zeta_B(w)&=1.
\end{align*}
Since $\zeta$ switches $\Sigma$ to all-positive, for edges $v_iv_j$,
	\begin{align*} \sigma_B^{\zeta_B}(v_iv_j)&=\zeta_B(v_i)\sigma_B(v_iv_j)\zeta_B(v_j)\\
	&=\zeta(v_i)\sigma(v_iv_j)\zeta(v_j)\\
	&=\sigma^\zeta(v_iv_j)\\
	&=+1.
\end{align*}
Similarly, for edges $v_i'v_j$,
	\begin{align*} \sigma_B^{\zeta_B}(v_i'v_j)&=\zeta_B(v_i')\sigma_B(v_i'v_j)\zeta_B(v_j)\\
	&=\zeta(v_i)\sigma(v_iv_j)\zeta(v_j)\\
	&=\sigma^\zeta(v_iv_j)\\
	&=+1.
\end{align*}
Also, for edges $v_i'w$,
	\begin{align*} \sigma_B^{\zeta_B}(v_i'w)&=\zeta_B(v_i')\sigma_B(v_i'w)\zeta_B(w)\\
	&=\zeta(v_i)\zeta(v_i)(+1)\\
	&=(\zeta(v_i))^2\\
	&=+1.
\end{align*}
Hence, $\zeta_B$ switches $M_B(\Sigma)$ to all-positive. Thus, $M_B(\Sigma)=(M(G),\sigma_B)$ is balanced, which we call as the balanced Mycielskian of $\Sigma$.
\begin{defn}[Balanced Mycielskian]
	Let $\Sigma=(G,\sigma)$ be a balanced signed graph, where the underlying graph $G=(V,E)$, is a finite simple  connected graph. The signed graph $M_B(\Sigma)=(M(G),\sigma_B)$ is called the balanced Mycielskian of $\Sigma$.
\end{defn}
One can observe that under this construction, if two balanced signed graphs $\Sigma_1$ and  $\Sigma_2$ are switching equivalent, then their corresponding balanced Mycielskians $M_B(\Sigma_1)$ and $M_B(\Sigma_2)$ are also switching equivalent.
\begin{rmk}
	Note that since $\sigma^\zeta(v_iv_j)=+1$, for every edge $v_iv_j$ in $\Sigma$, we have $\zeta(v_i)\zeta(v_j)=\sigma(v_iv_j)$. Thus,
	\begin{align*}
		\sigma_B(v_i'w)\sigma_B(v_i'w)&=\zeta(v_i)\zeta(v_j)\\
		&=\sigma(v_iv_j)
	\end{align*}
	Hence, the signature function defined for the balanced Mycielskian satisfies the condition given in Theorem~\ref{ag1}.
\end{rmk}
\begin{ex}
	Let $\Sigma$ be the balanced  4-cycle shown in Figure~\ref{fig:sub3}. The switching function $\zeta:V(\Sigma)\rightarrow\{-1,1\}$ defined by $\zeta(v_1)=\zeta(v_3)=\zeta(v_4)=-1$ and $\zeta(v_2)=1$ switches $\Sigma$ to all-positive. The corresponding balanced Mycielskian is constructed in Figure~\ref{fig:sub4}.
	\begin{figure}[h!]
		\centering
		\begin{subfigure}{.5\textwidth}
			\centering
			\begin{tikzpicture}[x=0.75cm, y=0.75cm]
				\vertex[fill] (v) at (0,0) [label=left:$v_1$] {};
				\vertex[fill] (w) at (4,0) [label=right:$v_2$] {};
				\vertex[fill] (x) at (0,4) [label=left:$v_4$] {};
				\vertex[fill] (y) at (4,4) [label=right:$v_3$] {};
				\path[very thick, dotted]
				(v) edge (w)
				(w) edge (y)

				;
				\path
				(x) edge (v)
				(y) edge (x)
				;
			\end{tikzpicture}
			\caption{$\Sigma$}
			\label{fig:sub3}
		\end{subfigure}%
		\begin{subfigure}{.5\textwidth}
			\centering
			\begin{tikzpicture}[x=0.75cm, y=0.75cm]
				\vertex[fill] (v1) at (0,0) [label=left:$v_1$] {};
				\vertex[fill] (v2) at (2,0) [label=above:$v_2$] {};
				\vertex[fill] (v3) at (4,0) [label=above:$v_3$] {};
				\vertex[fill] (v4) at (6,0) [label=right:$v_4$] {};
				\vertex[fill] (v1') at (0,-2) [label=left:$v_1'$] {};
				\vertex[fill] (v2') at (2,-2) [label=right:$v_2'$] {};
				\vertex[fill] (v3') at (4,-2) [label=right:$v_3'$] {};
				\vertex[fill] (v4') at (6,-2) [label=right:$v_4'$] {};
				\vertex[fill] (w) at (3,-4) [label=below:$w$] {};

				\path[very thick, dotted]
				(v1) edge (v2)
				(v1) edge (v2')
				(v1') edge (v2)
				(v2) edge (v3)
				(v2) edge (v3')
				(v2') edge (v3)
				(v1') edge (w)
				(v3') edge (w)
				(v4') edge (w)
				;
				\path
				(v3) edge (v4)
				(v3') edge (v4)
				(v3) edge (v4')
				(v1') edge (v4)
				(v1) edge (v4')
				(v2') edge (w)
				;
				\draw[] (0,0) .. controls (3,2.3) .. (6, 0);
			\end{tikzpicture}
			\caption{$M_B(\Sigma)$}
			\label{fig:sub4}
		\end{subfigure}
		\caption{A balanced signed graph $\Sigma$ and its balanced Mycielskian $M_B(\Sigma)$.}
		\label{fig:test2}
	\end{figure}
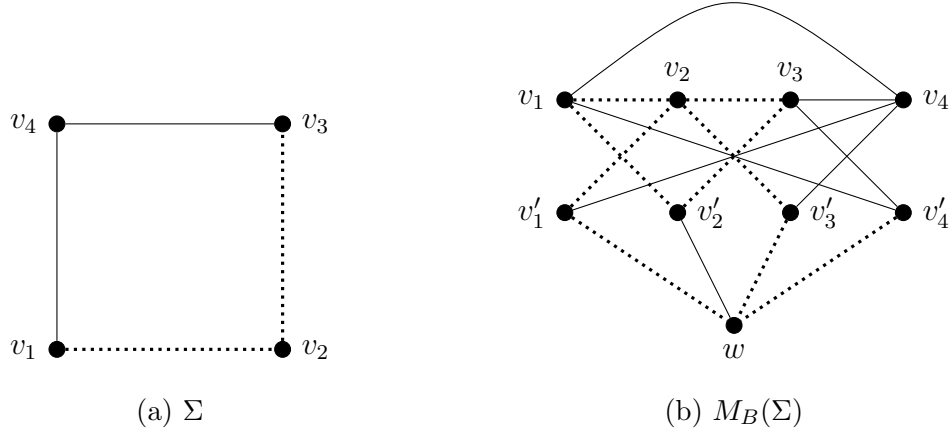
\end{ex}

\section{The chromatic number of Mycielskian of signed graphs}\label{sec3}
In 1981, Zaslavsky \cite{tz2} introduced the concept of coloring a signed graph. For a signed graph $\Sigma$, he defined the signed coloring of $\Sigma$ in $\mu$ colors, or in $2\mu+1$ signed colors as a mapping  $c:V(\Sigma)\rightarrow\{-\mu,-\mu+1,\dots,0,\dots,\mu-1,\mu\}$. Whenever a coloring never assumes the value 0, it is referred to as a zero-free coloring. A coloring $c$ is said to be proper if $c(u)\neq\sigma(e)c(v)$ for every edge $e=uv$ of $\Sigma$ (see \cite[Section 1]{tz2}).

Máčajová \textit{et al.} in \cite{em} defined the chromatic number of a signed graph as follows.
\begin{defn}\cite{em}\label{def2}
	An $n$ - coloring of a signed graph $\Sigma$ is a proper coloring that uses colors from the set $M_n$, which is defined for each $n\geq1$ as
	$$M_n=\begin{cases}
		\{\pm 1,\pm 2,\dots\,\pm k\} & \text{if} \,\,  n=2k\\
		\{0, \pm 1,\pm 2,\dots\,\pm k\} & \text{if}\,\,   n=2k+1
	\end{cases}$$
	The smallest $n$ such that $\Sigma$ admits an $n$ - coloring is called the \textit{chromatic number} of $\Sigma$ and is denoted by $\chi(\Sigma)$.

	The chromatic number of a balanced signed graph coincides with the chromatic number of its underlying unsigned graph.
\end{defn}
\begin{prop}
 Let $M(\Sigma)\backslash\{w\}$ be the signed graph obtained by removing the root vertex $w$ (and the corresponding edges) from $M(\Sigma)$. Then $\chi(M(\Sigma)\backslash\{w\})=\chi(\Sigma)$.
\end{prop}
\begin{proof}
	Let $\chi(\Sigma)=n$ and let $c:V(\Sigma)\rightarrow M_n$ be an $n$ - coloring for $\Sigma$. Define $c':V((M(\Sigma)\backslash\{w\})\rightarrow M_n$ by $c'(v_i')=c'(v_i)=c(v_i)$ for all $i$. Since $c(v_i)\neq\sigma(v_iv_j)c(v_j)$, it follows that $c'(v_i)\neq\sigma_M(v_iv_j)c'(v_j)$ and $c'(v_i')\neq\sigma_M(v_i'v_j)c'(v_j)$. Hence, $c'$ is an $n$ - coloring for $M(\Sigma)\backslash\{w\}$.
\end{proof}
For any given signed graph $\Sigma$, there exist a signed graph $-\Sigma$ obtained by reversing the signs of all edges of $\Sigma$. We say $\Sigma$ is antibalanced when $-\Sigma$ is balanced. Note that $\Sigma$ is antibalanced if and only if it can be switched to all-negative.\\

	We restate the Lemma 2.4  from \cite{tz4} as follows.
\begin{lem}[\cite{tz4}]\label{tz4}
A signed graph $\Sigma$ is antibalanced if and only if $\chi(\Sigma)\leq 2$.
\end{lem}
\begin{thm}
	Let $\Sigma$ be a signed graph and $M(\Sigma)$ be its Mycielskian. Then, $\chi(M(\Sigma)) \leq 2$ if and only if $\Sigma$ is all-negative.
\end{thm}
\begin{proof}
	If $\Sigma$ is an all-negative signed graph with vertex set $\{v_1,v_2,\dots v_p\}$, then the only positive edges of $M(\Sigma)$ are $v_i'w$, $1\leq i \leq p$. Now, the switching function $\zeta_M':V(M(\Sigma))\rightarrow \{-1,1\}$ defined by $\zeta_M'(v_i)=\zeta_M'(v_i')=1$ for all $1\leq i \leq p$ and $\zeta_M'(w)=-1$ switches $M(\Sigma)$ to all-negative. Therefore, $M(\Sigma)$ is antibalanced and hence $\chi(M(\Sigma)) \leq 2$, by Lemma~\ref{tz4}. Conversely, if $\Sigma$ is not all-negative, it contains at least one positive edge, say $v_iv_j$. Then $v_iv_jv_i'w  v_j'v_i$ forms a negative $5$ - cycle in $-M(\Sigma)$, making it unbalanced. Thus, $M(\Sigma)$ is not antibalanced and therefore, by Lemma~\ref{tz4},  $\chi(M(\Sigma))>2$.
\end{proof}

We have the following theorem in \cite{br}.
\begin{thm}[\cite{br}]\label{br1}
	Let $\chi(G)$ and $\chi(M(G))$ be the chromatic numbers of a graph $G$ and its Mycielskian $M(G)$ respectively. Then
	$\chi(M(G))=\chi(G)+1$.
\end{thm}

\begin{thm}\label{ag3}
	Let $M(\Sigma)$ be the  Mycielskian  of a signed graph $(\Sigma)$. Then,  $\chi(\Sigma)\leq\chi(M(\Sigma))\leq\chi(\Sigma)+1$.
	Furthermore, $\chi(M(\Sigma))=\chi(\Sigma)$ if $\Sigma$ is all-negative and $\chi(M(\Sigma))=\chi(\Sigma)+1$ if $\Sigma$ is all-positive.
\end{thm}
\begin{proof}
	Let $\chi(\Sigma)=n$ and let $c:V\rightarrow M_n$ be an $n$ - coloring for $\Sigma$. We extend $c$ to an $(n+1)$ - coloring of $M(\Sigma)$. If $n=2k$, we extend $c$ to an $(n+1)$ - coloring of $M(\Sigma)$ by setting $c(v_i')=c(v_i)$ for all $i$  and $c(w)=0$. If $n=2k+1$, we extend $c$ to an $(n+1)$ - coloring of $M(\Sigma)$ as follows. Let $v_t$ be any vertex in $V$ with $c(v_t)=0$. Then for all $v_i\neq v_t$, set $c(v_i')=c(v_i)$ , $c(v_t')=c(v_t)=k+1$ and $c(w)=-(k+1)$. Hence, $\chi(M(\Sigma))\leq \chi(\Sigma)+1$.

	Now, if $\Sigma$ is all-negative, it can be colored using just one color, namely $-1$. Let $c:V(\Sigma)\rightarrow \{\pm1\}$ be the proper $2$ - coloring for $\Sigma$. This can be extended to a proper $2$ - coloring for $M(\Sigma)$ by setting $c(v_i')=c(v_i)=-1$ for all $i$  and $c(w)=+1$. If $\Sigma$ is all-positive, then $M(\Sigma)$ is all-positive. Thus, $\chi(M(\Sigma))=\chi(|M(\Sigma)|)=\chi(|\Sigma|)+1=\chi(\Sigma)+1$.
\end{proof}

\begin{rmk}
	Let $\Sigma$ be a signed graph with $\chi(\Sigma)=n$ and let $c:V(\Sigma)\rightarrow M_n$ be an $n$ - coloring of $\Sigma$. The \textit{deficiency} of the coloring c is the number of unused colors from $M_n$ (see \cite{arwm}).
	The existence of signed graphs satisfying $\chi(M(\Sigma))=\chi(\Sigma)$ is a consequence of the deficiency of the coloring of $\Sigma$. Specifically, if the coloring of $\Sigma$ has deficiency at least $1$, then an unused color can be assigned to $w$, making the chromatic number of  $M(\Sigma)$ and $\Sigma$ equal. As an example, consider $\Sigma$ as the balanced $3$ - cycle shown in Figure~\ref{fig:sub9}. Note that $\chi(\Sigma)=3$ and the color $-1$ in the color set $\{0,\pm 1\}$ is unused.
	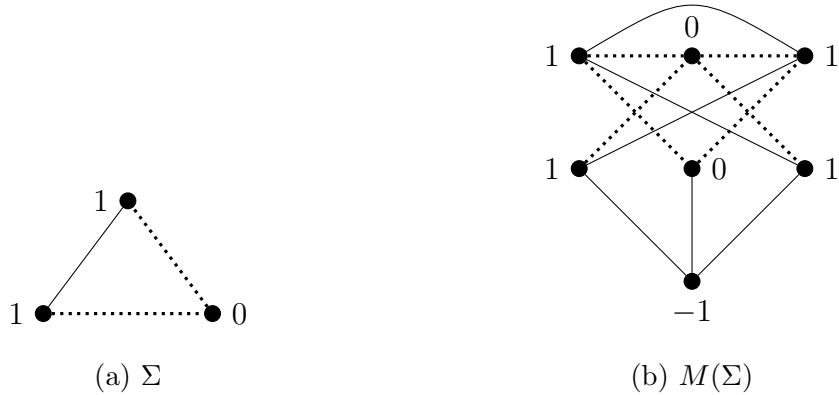
\begin{figure}[h!]
	\centering
	\begin{subfigure}{.5\textwidth}
		\centering
		\begin{tikzpicture}[x=0.75cm, y=0.75cm]
			\vertex[fill] (v1) at (0,0) [label=left:$1$] {};
			\vertex[fill] (v2) at (3,0) [label=right:$0$] {};
			\vertex[fill] (v3) at (1.5,2) [label=left:$1$] {};
			\path[very thick, dotted]
			(v1) edge (v2)
			(v2) edge (v3)

			;
			\path
			(v1) edge (v3)
			;
		\end{tikzpicture}
		\caption{$\Sigma$}
		\label{fig:sub9}
	\end{subfigure}%
	\begin{subfigure}{.5\textwidth}
		\centering
		\begin{tikzpicture}[x=0.75cm, y=0.75cm]
			\vertex[fill] (v1) at (0,0) [label=left:$1$] {};
			\vertex[fill] (v2) at (2,0) [label=above:$0$] {};
			\vertex[fill] (v3) at (4,0) [label=right:$1$] {};
			\vertex[fill] (v1') at (0,-2) [label=left:$1$] {};
			\vertex[fill] (v2') at (2,-2) [label=right:$0$] {};
			\vertex[fill] (v3') at (4,-2) [label=right:$1$] {};
			\vertex[fill] (w) at (2,-4) [label=below:$-1$] {};

			\path[very thick, dotted]
			(v1) edge (v2)
			(v1) edge (v2')
			(v1') edge (v2)
			(v2) edge (v3)
			(v2) edge (v3')
			(v2') edge (v3)

			;
			\path

			(v1) edge (v3')
			(v1') edge (v3)
			(v1') edge (w)
			(v2') edge (w)
			(v3') edge (w)
			;
			\draw[] (0,0) .. controls (2,1.2) .. (4, 0);
		\end{tikzpicture}
		\caption{$M(\Sigma)$}
		\label{fig:sub10}
	\end{subfigure}
	\caption{A signed graph $\Sigma$  satisfying $\chi(M(\Sigma))=\chi(\Sigma)$}
	\label{fig:test4}
\end{figure}
\end{rmk}
We now establish some results on the balanced Mycielskian of signed graphs.
\begin{prop}\label{p1}
Let $\Sigma=(G,\sigma)$ be a balanced signed graph and $M_B(\Sigma)=(M(G),\sigma_B)$ be its balanced Mycielskian. Then $\chi(M_B(\Sigma))=\chi(\Sigma)+1$.
\end{prop}
\begin{proof}
	Since $\Sigma$ and $M_B(\Sigma)$ are both balanced, $\chi(M_B(\Sigma))=\chi(M(G))$ and $\chi(\Sigma)=\chi(G)$. The result then follows from Theorem~\ref{br1}.
\end{proof}
The following theorem was put forward by Mycielski in \cite{my}
\begin{thm}[\cite{my}]\label{my}
For any positive integer $n$, there exists a
triangle-free graph with chromatic number $n$.
\end{thm}
The next theorem is an analogous result for balanced signed graphs.
\begin{thm}\label{ag5}
	For any positive integer $n$, there exists a balanced
	triangle-free signed graph that is not all-positive, and  having chromatic number $n$.
\end{thm}
\begin{proof}
	The proof is based on mathematical induction. For $n=1$ and $n=2$, the signed graphs $\Sigma_1=K_1$ and $\Sigma_2=K_2^-$, where $K_2^-$ is the all-negative signed complete graph on two vertices have the required property. Suppose that for $k>2$, such a signed graph $\Sigma_k$ satisfying the induction hypothesis exist. Then $M_B(\Sigma_k)$ is a balanced signed graph that is not all-positive. Also, by Proposition~\ref{p1}, we have, $\chi(\Sigma_{k+1})=\chi(\Sigma_{k})+1=k+1$.
\end{proof}
The first four signed graphs mentioned in Theorem~\ref{ag5} are shown in Figure~\ref{fig:test3}.
	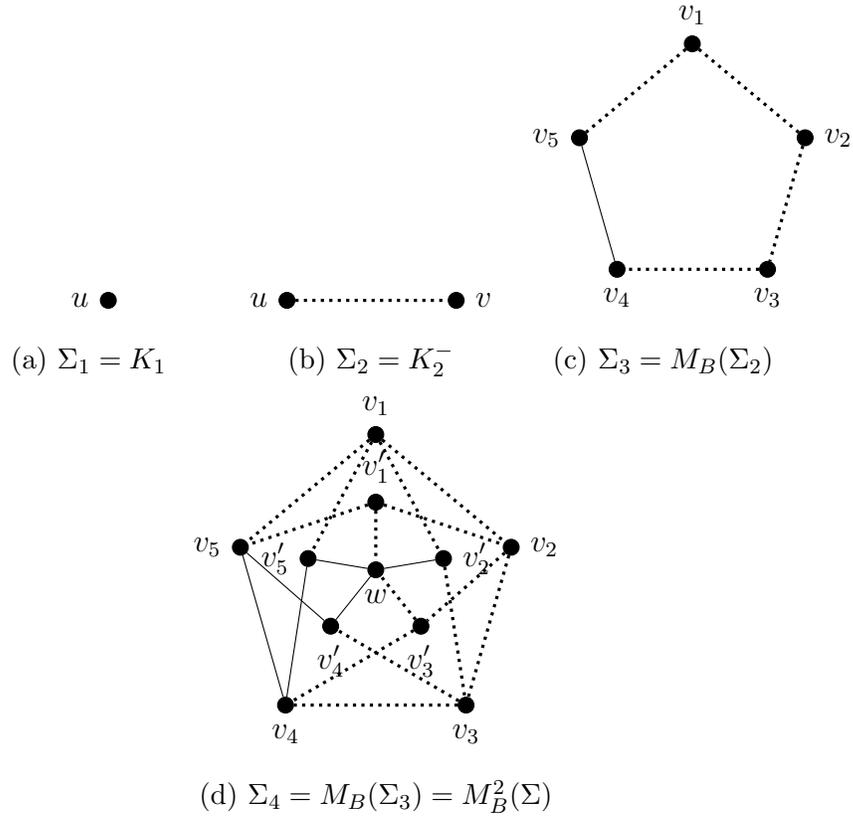
\begin{figure}[h!]
		\centering
		\begin{subfigure}{.25\textwidth}
			\centering
			\begin{tikzpicture}[x=0.6cm, y=0.6cm]
				\vertex[fill] (u) at (0,0)
				[label=left:$u$] {};
				;
			\end{tikzpicture}
			\caption{$\Sigma_1=K_1$}
			\label{fig:sub5}
		\end{subfigure}%
	\begin{subfigure}{.25\textwidth}
			\centering
			\begin{tikzpicture}[x=0.75cm, y=0.75cm]
				\vertex[fill] (u) at (0,0)
				 [label=left:$u$] {};
				 \vertex[fill] (v) at (3,0) [label=right:$v$] {};
				\path[very thick, dotted]
				(u) edge (v)
				;
			\end{tikzpicture}
			\caption{$\Sigma_2=K_2^-$}
			\label{fig:sub6}
		\end{subfigure}
		\begin{subfigure}{.25\textwidth}
			\centering
	\begin{tikzpicture}[x=0.5cm, y=0.5cm]
	\vertex[fill] (v1) at (0,2.5) [label=above:$v_1$] {};
	\vertex[fill] (v2) at (3,0) [label=right:$v_2$] {};
	\vertex[fill] (v3) at (2,-3.5) [label=below:$v_3$] {};
	\vertex[fill] (v4) at (-2,-3.5) [label=below:$v_4$] {};
	\vertex[fill] (v5) at (-3,0) [label=left:$v_5$] {};
	\path[very thick, dotted]
	(v1) edge (v2)
	(v2) edge (v3)
	(v3) edge (v4)
	(v1) edge (v5)
	;
	\path
	(v4) edge (v5)

	;
\end{tikzpicture}
\caption{$\Sigma_3=M_B(\Sigma_2)$}
\label{fig:sub7}
\end{subfigure}\\
\begin{subfigure}{.5\textwidth}
			\centering
			\begin{tikzpicture}[x=0.6cm, y=0.6cm]
				\vertex[fill] (v1) at (0,2.5) [label=above:$v_1$] {};
				\vertex[fill] (v2) at (3,0) [label=right:$v_2$] {};
				\vertex[fill] (v3) at (2,-3.5) [label=below:$v_3$] {};
				\vertex[fill] (v4) at (-2,-3.5) [label=below:$v_4$] {};
				\vertex[fill] (v5) at (-3,0) [label=left:$v_5$] {};
				\vertex[fill] (v1') at (0,1) [label=above:$v_1'$] {};
				\vertex[fill] (v2') at (1.5,-0.25) [label=right:$v_2'$] {};
				\vertex[fill] (v3') at (1,-1.75) [label=below:$v_3'$] {};
				\vertex[fill] (v4') at (-1,-1.75) [label=below:$v_4'$] {};
				\vertex[fill] (v5') at (-1.5,-0.25) [label=left:$v_5'$] {};
				\vertex[fill] (w) at (0,-0.5) [label=below:$w$] {};

				\path[very thick, dotted]
				(v1) edge (v2)
				(v2) edge (v3)
				(v3) edge (v4)
				(v1) edge (v5)
				(v1') edge (v2)
				(v2') edge (v3)
				(v3') edge (v4)
				(v1') edge (v5)
				(v1) edge (v2')
				(v2) edge (v3')
				(v3) edge (v4')
				(v1) edge (v5')
				(v1') edge (w)
				(v3') edge (w)
				;
				\path
				(v4) edge (v5)
				(v4') edge (v5)
				(v4) edge (v5')
				(v2') edge (w)
				(v4') edge (w)
				(v5') edge (w)
				;
			\end{tikzpicture}
			\caption{$\Sigma_4=M_B(\Sigma_3)=M^2_B(\Sigma)$}
			\label{fig:sub8}
		\end{subfigure}
		\caption{Iterated balanced Mycielskians.}
		\label{fig:test3}
	\end{figure}
\section{Matrices of the Mycielskian of signed graphs}\label{sec4}

Given a signed graph $\Sigma=(V,E,\sigma)$ where  $V = \{v_1,v_2,\dots,v_p\}$ is the vertex set, $E=\{e_1,e_2,\dots,e_q\}$ is the edge set and $\sigma:E\rightarrow \{-1,1\}$ is the sign function.
Let $M(\Sigma)$ be the Mycielskian of $\Sigma$.  In this section, we introduce the adjacency matrix, the incidence  matrix and the Laplacian matrix of the Mycielskian $M(\Sigma)$ of $\Sigma$.
\subsection{ The adjacency matrix}

The adjacency matrix of $\Sigma$, denoted by $\textbf{A}=\textbf{A}(\Sigma)$, is a $p\times p$ matrix $(a_{ij})$ in which $a_{ij}=\sigma(v_iv_j)$ if $v_i$ and $v_j$ are adjacent and $0$ otherwise (see \cite[Section 3]{tz1}).

Since $v_i$ is adjacent to $v_j'$ and $v_i'$ is adjacent to $v_j$ in $M(\Sigma)$ whenever $v_i$ and $v_j$ are adjacent in $\Sigma$, \textbf{the adjacency matrix} $\textbf{A}_\textbf{M}$  of the Mycielskian  $M(\Sigma)$ takes the block form $$\textbf{A}_\textbf{M}=\textbf{A}(M(\Sigma))=
\begin{bmatrix}
	\textbf{A}(\Sigma)&\textbf{A}(\Sigma) & \textbf{0}_{p\times1} \\
	\textbf{A}(\Sigma)&  \textbf{0}_{p\times p}&  \textbf{j}_{p\times1}\\
	\textbf{0}^t_{1\times p}&  \textbf{j}^t_{1\times p}& 0
\end{bmatrix}$$ where $\textbf{0}$ is a matrix of zeros and $\textbf{j}$ is a matrix of ones of the specified order.\\
 $\textbf{A}_\textbf{M}$ is a symmetric matrix of order $2p+1$.

Given a graph $G$ with adjacency matrix $A(G)$, the connection between the ranks of $A(G)$ and $A(M(G))$, the connection between the number of positive, negative and zero eigenvalues $A(G)$ and $A(M(G))$ were studied by Fisher \textit{et al.} in \cite{fmb}. We initiate a similar study in the case of signed graphs.

 	Let $\Sigma=(V,E,\sigma)$ be a given signed graph and let $t\notin V$. We denote the signed graph obtained by joining all the vertices of $\Sigma$ to $t$ with negative edges by $\Sigma_{t^-}$. That is, $\Sigma_{t^-}$ is the negative join $\Sigma\vee_- K_1$. The adjacency matrix of $\Sigma_t$ takes the block form
 	$$\textbf{A}_{t^-}=\textbf{A}(\Sigma_{t^-})=\begin{bmatrix}
 		\textbf{A}&\textbf{-j}  \\
 		\textbf{-j}^t& 0
 	\end{bmatrix}$$
 We now have the following theorem.
  \begin{thm}
 	Let $\Sigma$ be a signed graph and $\textbf{A}(\Sigma)$ be the adjacency matrix of $\Sigma$. Let  $r(\textbf{A})$ denote the rank and $n_+(\textbf{A})$, $n_-(\textbf{A})$ and $n_0(\textbf{A})$  respectively denote the number of positive, negative and zero eigenvalues of a symmetric matrix \textbf{A}, then we have the following.
 	\begin{itemize}
 		\item [$(i)$] $r(\textbf{A}_\textbf{M})=r(\textbf{A})+r(\textbf{A}_{t^-})$
 		\item [$(ii)$] $n_+(\textbf{A}_\textbf{M})=n_+(\textbf{A})+n_+(\textbf{A}_{t^-})$
 		\item [$(iii)$] $n_-(\textbf{A}_\textbf{M})=n_-(\textbf{A})+n_-(\textbf{A}_{t^-})$
 		\item [$(iv)$] $n_0(\textbf{A}_\textbf{M})=n_0(\textbf{A})+n_0(\textbf{A}_{t^-})$
 	\end{itemize}
 \end{thm}
\begin{proof}
	 The adjacency matrix $\textbf{A}_\textbf{M}$ can be factorized as
	$$\textbf{A}_\textbf{M}=
	\begin{bmatrix}
		\textbf{A}&\textbf{A}& \textbf{0} \\
		\textbf{A}& \textbf{0}&  \textbf{j}\\
		\textbf{0}^t&  \textbf{j}^t& 0
	\end{bmatrix} = \begin{bmatrix}
		\textbf{I}& \textbf{0} & \textbf{0}\\
		\textbf{I}&  \textbf{-I}& \textbf{0}\\
		\textbf{0}&  \textbf{0}^t& 1
	\end{bmatrix}  \begin{bmatrix}
		\textbf{A}&\textbf{0}& \textbf{0}\\
		\textbf{0}& \textbf{-A}&  \textbf{-j}\\
		\textbf{0}^t&  \textbf{-j}^t& 0
	\end{bmatrix}  \begin{bmatrix}
		\textbf{I}&\textbf{I}& \textbf{0}\\
		\textbf{0}&  \textbf{-I}&  \textbf{0}\\
		\textbf{0}^t&  \textbf{0}^t& 1
	\end{bmatrix}=\textbf{P} \, \textbf{B} \, \textbf{P}^t$$

	where, $\textbf{P}=\begin{bmatrix}
		\textbf{I}& \textbf{0} & \textbf{0}\\
		\textbf{I}&  \textbf{-I}& \textbf{0}\\
		\textbf{0}&  \textbf{0}^t& 1
	\end{bmatrix}$ is an invertible matrix and $\textbf{B}=\begin{bmatrix}
		\textbf{A}& \textbf{0} \\
		\textbf{0}& \textbf{A}_{t^-}
	\end{bmatrix}$.

	Thus, the matrices $\textbf{A}_\textbf{M}$ and $\textbf{B}$ are congruent, and hence by Sylvester's law of inertia, they have the same rank and the same number of positive, negative and zero eigenvalues.
\end{proof}
 \subsection{The incidence matrix}
The incidence matrix of $\Sigma$, denoted by $\textbf{H}=\textbf{H}(\Sigma)$, is  the $p\times q$ matrix
$$\textbf{H}(\Sigma)=\begin{bmatrix}
	\textbf{x}(e_1)& \textbf{x}(e_2) &  \cdots&\textbf{x}(e_q)
\end{bmatrix}$$ where, for each edge $e_k=v_iv_j$, $1\leq k\leq q$, the vector $\textbf{x}(e_k)=\begin{pmatrix}
	x_{1k}\\
	\vdots	\\
	x_{pk}
\end{pmatrix} \in \mathbb{R}^{p\times1}$ has its $i^{\text{th}}$ and $j^{\text{th}}$ entries as $x_{ik}=\pm 1$ and $x_{jk}=\mp
\sigma(e_k)$ respectively and  all other entries as $0$ (see \cite[Section 3]{tz1}).

Let us denote  the vertex set $V_M$ and the edge set $E_M$ of $M(\Sigma)$ as $$V_M=\{v_1,v_2,\dots, v_p,v_1', v_2',\dots, v_p',w\}$$ $$E_M=\{e_1,e_2, \dots, e_q,e_1',e_1'',e_2',e_2'',\dots, e_q',e_q'',f_1,f_2,\cdots,f_p \}$$ respectively, where, for each $1\leq k\leq q$, the edges $e_k'$ and $e_k''$ of $M(\Sigma)$ are defined by $e_k'=v_iv_j'$ and $e_k''=v_i'v_j$ whenever $e_k=v_iv_j$ is an edge of $\Sigma$ with $1\leq i < j\leq q$ and $f_i$ is defined by $f_i=v_i'w$ for $1\leq i \leq p$. Then, the \textbf{incidence matrix} $\textbf{H}_\textbf{M}=\textbf{H}(M(\Sigma))$ takes the block form
$$
\textbf{H}_\textbf{M}=\textbf{H}(M(\Sigma))=\left[
\begin{array}{c|c|c|c|c|c|c|c|c}
	\textbf{H}(\Sigma)_{p\times q} &\textbf{x}_\textbf{1}&\textbf{y}_\textbf{1}&\textbf{x}_\textbf{2}&\textbf{y}_\textbf{2}&\cdots&\textbf{x}_\textbf{p}&\textbf{y}_\textbf{p} & \textbf{0}_{p\times p}\\ \hhline{--------|-}
	\textbf{0}_{p\times q} &\textbf{y}_\textbf{1}&\textbf{x}_\textbf{1}&\textbf{y}_\textbf{2}&\textbf{x}_\textbf{2}&\cdots&\textbf{y}_\textbf{p}&\textbf{x}_\textbf{p} & \textbf{I}_{p\times p}\\ \hhline{--------|-}
	\multicolumn{8}{c|}{\textbf{0}_{1\times 3q}} &\textbf{-j}_{1\times p} \\ \hhline{|~}
\end{array}\right]
$$

Here, $	\textbf{H}(\Sigma)$ is the incidence matrix of $\Sigma$, $\textbf{I}$ is the identity matrix, $\textbf{0}$ is the zero matrix and $\textbf{-j}$ is the matrix with all entries $-1$ of the specified order.   $\textbf{x}_\textbf{i}$'s and $\textbf{y}_\textbf{i}$'s are matrices of order $p\times 1$ and satisfies the condition $\textbf{x}_\textbf{i}+\textbf{y}_\textbf{i}=\textbf{x}(e_i)$ for all $1\leq i\leq q$.

\subsection{The Laplacian matrix}
The Laplacian matrix of $\Sigma$,  denoted by $\textbf{L}=\textbf{L}(\Sigma)$ is the $p\times p$ matrix   $$\textbf{L}(\Sigma) = \textbf{D}(| \Sigma|)-\textbf{A}(\Sigma)$$ where $\textbf{A}(\Sigma)$ is the adjacency matrix of $\Sigma$ and $\textbf{D}(|\Sigma|)$ is the degree matrix of the underlying graph $|\Sigma|$ (see \cite[Section 3]{tz1}).

Accordingly, we define the \textbf{Laplacian matrix for the Mycielskian} of $\Sigma$  as
$$\textbf{L}_\textbf{M}=\textbf{L}(M(\Sigma))= \textbf{D}(|M(\Sigma)|)-\textbf{A}(M(\Sigma))=\textbf{D}_\textbf{M}-\textbf{A}_\textbf{M}$$
where, $\textbf{A}_\textbf{M}$ is the adjacency matrix and $\textbf{D}_\textbf{M}$ is the diagonal degree matrix of the Mycielskian of $\Sigma$. Now, $\textbf{D}_\textbf{M}$ takes the block form
$$\textbf{D}_\textbf{M}=\begin{bmatrix}
	2\textbf{D}(|\Sigma|)_{p\times p}& \textbf{0}_{p\times p}  &\textbf{0}_{p\times 1}  \\
	\textbf{0}_{p\times p}&(\textbf{D}(|\Sigma|)+\textbf{I})_{p\times p}  &\textbf{0}_{p\times 1}  \\
	\textbf{0}^t_{1\times p}& \textbf{0}^t_{1\times p} & p
\end{bmatrix}$$
where $p=|V|$, $\textbf{D}(\Sigma)$ is the diagonal degree matrix of $\Sigma$, \textbf{I} is the identity matrix and  \textbf{0} is the zero matrix of the specified order.

Consequently, the Laplacian matrix $\textbf{L}_\textbf{M}=\textbf{L}(M(\Sigma))$ takes the block form $$\textbf{L}_\textbf{M}=\begin{bmatrix}
	(2\textbf{D}(|\Sigma|)-\textbf{A}(\Sigma))_{p\times p}	& \textbf{-A}(\Sigma)_{p\times p} & \textbf{0}_{p\times 1}  \\
	\textbf{-A}(\Sigma)_{p\times p}&(\textbf{D}(|\Sigma|)+\textbf{I})_{p\times p}  & \textbf{-j}_{p\times 1}  \\
	\textbf{0}^t_{1\times p}&  \textbf{-j}^t_{1\times p}& p
\end{bmatrix}$$

\section{Conclusion and Scope}
In this paper, we have defined the Mycielskian of a signed graph and discussed some of its properties. We have seen that the Mycielskian of a balanced signed graph need not be balanced and hence we provide an alternate construction in which the Mycielskian of $\Sigma$ is balanced whenever $\Sigma$ is balanced, This paper also discuss the chromatic number of the Mycielskian of a signed graph and established that the chromatic number of a signed graph and its Mycielskian are related.  We also established the block forms of various matrices of the Mycielskian of a signed graph such as the adjacency matrix, the incidence matrix and the Laplacian matrix. Developing another balance preserving, switching preserving  constructions for Mycielskian of  signed graphs, computing the spectrum of the Mycielskian of signed graphs for various matrices are some interesting areas for further investigation.
\section*{Acknowledgments}
The first author would like to acknowledge his gratitude to the University Grants Commission (UGC), India, for providing financial support in the form of Junior Research fellowship (NTA Ref. No.: 191620039346). The authors express their sincere gratitude to Professor Thomas Zaslavsky, Binghamton University (SUNY), Binghamton, for valuable discussions and guidance throughout the process.

\end{document}